\documentclass[11pt,a4paper]{amsart}%
\usepackage{amssymb}
\usepackage{graphicx}
\usepackage{algorithm}
\usepackage{algorithmic}
\usepackage{pstricks,pst-node,pst-plot}
\usepackage[centertags]{amsmath}
\usepackage{amsfonts}
\usepackage{amsthm}
\usepackage{newlfont}
\usepackage{tikz}%
\providecommand{\U}[1]{\protect\rule{.1in}{.1in}}
\theoremstyle{plain}
\newtheorem{theorem}{Theorem}[section]
\newtheorem{corollary}[theorem]{Corollary}

\theoremstyle{definition}

\theoremstyle{remark}

\newtheorem{example}{Example}[section]

\numberwithin{equation}{section}

\numberwithin{table}{section}

\numberwithin{figure}{section}
\begin{document}
\title[An overview of $(\kappa,\tau)$-regular sets and their applications]{\sc An overview of $(\kappa,\tau)$-regular sets and their applications}
\author{Domingos M. Cardoso}
\address{Center for Research and Development in Mathematics and Applications,
Department of Mathematics, University of Aveiro, 3810-193 Aveiro, Portugal}
\email{dcardoso@ua.pt}
\date{\today }
%\date{}
%\date{\today }
%\date{}
\maketitle

\begin{abstract}
A $(\kappa,\tau)$-regular set is a vertex subset $S$ inducing a $\kappa$-regular subgraph such that every vertex out of $S$ has $\tau$
neighbors in $S$. This article is an expository overview of the main results obtained for graphs with $(\kappa,\tau)$-regular sets. The
graphs with classical combinatorial structures, like perfect matchings, Hamilton cycles, efficient dominating sets, etc, are characterized
by $(\kappa, \tau)$-sets whose determination is equivalent to the determination of those classical combinatorial structures. The
characterization of graphs with these combinatorial structures are presented. The determination of $ (\kappa,\tau) $-regular sets in a finite
number of steps is deduced and the main spectral properties of these sets are described.
\end{abstract}

\bigskip


\bigskip

\section{Introduction}
The graphs $G$ considered in this work are simple and undirected with \textit{order} (number of vertices) $n$ and \textit{size} (number of edges) $m$.
The vertex set of $G$ is denoted by $V(G)$ and its edge set by $E(G)$. An edge with \textit{end-vertices} i and j is denoted by $ij$. In such a case,
we say that these vertices are \textit{adjacent}. The \textit{neighborhood} of a vertex $u \in V(G)$ is $N_G(u) = \{v \in V(G): uv \in E(G)\}$ and the
\textit{degree} of the vertex $u$ is $d_G(u)=|N_G(u)|$. The \textit{maximum} (\textit{minimum}) \textit{degree} of the vertices of a graph $G$ is
$\Delta(G)$ ($\delta(G)$). Several other concepts and notation will be introduced throughout the text and for further basic notation and basic concepts the reader is referred to \cite{BH2012} and \cite{CRS10}.

The concept of $(\kappa,\tau)$-regular set for general graphs first appeared in \cite{CardosoRama04} as a particular case of the concept of $(X,Y)$-\textit{set}
introduced in \cite{telle93} as a vertex subset $S$ of a graph $G$ such that
$$
|N_{G}(v) \cap S| \in \left\{\begin{array}{ll}
                                    X, & \hbox{if } \; v \in S\\
                                    Y, & \hbox{otherwise,}
                             \end{array}\right.
$$
where $X$ and $Y$ are subsets of $\{0, 1, \cdots , n-1\}$. The vertex subset $S$ is $(\kappa,\tau)$-\textit{regular} when the subsets $X$ and
$Y$ are the singleton subsets $X = \{k\}$ and $Y = \{\tau\}$, that is, $S$ is $(\kappa,\tau)$-regular if $\forall v \in V(G)$
$$
|N_G(v) \cap S| = \left\{\begin{array}{ll}
                                 \kappa, & \hbox{if } v \in S\\
                                 \tau,   & \hbox{otherwise.}
                          \end{array}\right.
$$
When a graph $G$ is $\kappa$-regular, for convenience, its vertex set $V(G)$ is consider a $(\kappa,0)$-regular set. From now on,
for any graph $G$, it is assumed that a $(\kappa,\tau)$-regular set $S \subset V(G)$ (that is, $S$ is strictly included in $V(G)$)
is such that $\tau>0$.

\begin{example}
The Pertersen graph depicted in Figure~\ref{PetersenGraph} has several $(k,\tau)$-regular sets.
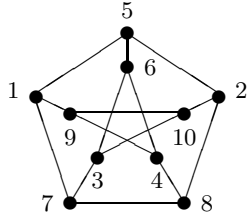
\begin{figure}[ht]
\begin{center}
\unitlength=0.3 mm
\begin{picture}(260,80)(120,90)
%
%Arestas do grafo de Petersen
\put(235,110){\line(2,1){55}}  %3-10-2
\put(265,110){\line(-2,1){55}} %4-9-1
\put(250,150){\line(-1,-3){14}}%6-3
\put(250,150){\line(1,-3){14}} %6-4
\put(225,90){\line(1,0){50}}   %7-8
\put(225,90){\line(-1,3){16}}  %7-1
\put(225,90){\line(3,5){12}}   %7-3
\put(275,90){\line(1,3){16}}   %8-2
\put(275,90){\line(-3,5){12}}  %8-4
\put(225,130){\line(1,0){50}}  %9-10
\put(250,165){\line(-3,-2){40}}%5-1
\put(250,165){\line(0,-1){15}} %5-6
\put(250,165){\line(3,-2){40}} %5-2
%
%vertices de Gamma1
\put(237,110){\circle*{5.7}} % {3}
\put(263,110){\circle*{5.7}} % {4}
\put(225,90){\circle*{5.7}} % {7}
\put(275,90){\circle*{5.7}} % {8}
\put(250,165){\circle*{5.7}}% {5}
\put(250,150){\circle*{5.7}}% {6}
\put(210,137){\circle*{5.7}}% {1}
\put(290,137){\circle*{5.7}}% {2}
\put(225,129){\circle*{5.7}}% {9}
\put(275,129){\circle*{5.7}}% {10}
%
%Etiquetacao dos vertices de Gamma1
%
\put(237,100){\makebox(0,0){\footnotesize 3}}
\put(263,100){\makebox(0,0){\footnotesize 4}}
\put(225,119){\makebox(0,0){\footnotesize 9}}
\put(275,119){\makebox(0,0){\footnotesize 10}}
\put(260,150){\makebox(0,0){\footnotesize 6}}
\put(215,90){\makebox(0,0){\footnotesize 7}}
\put(285,90){\makebox(0,0){\footnotesize 8}}
\put(250,175){\makebox(0,0){\footnotesize 5}}
\put(200,140){\makebox(0,0){\footnotesize 1}}
\put(300,140){\makebox(0,0){\footnotesize 2}}
\end{picture}
\caption{The Petersen graph.} \label{PetersenGraph}
\end{center}
\end{figure}
For instance, $S= \left\{\begin{array}{ll}
                             \{1,2,3,4\},      & \hbox{ is } $(0,2)$-regular;\\
                             \{5,6,7,8,9,10\}, & \hbox{ is } $(1,3)$-regular;\\
                             \{ 1,2, 5,7,8 \}, & \hbox{ is } $(2,1)$-regular.
                        \end{array}
                      \right.$
\end{example}

The following properties are immediate.
\begin{enumerate}
\item If an arbitrary graph $H$ has a $(\kappa,\tau)$-regular set $S \subset V(H)$, then $S$ is a $(|S|-\kappa-1,|S|-\tau)$-regular set for
      the complement graph of $H$.
\item If a $p$-regular graph $G$ has a $(\kappa,\tau)$-regular set $S \subset V(H)$, then  $V(G) \setminus S$ is $(p-\tau,p-\kappa)$-regular.
\end{enumerate}

Now we recall that the \emph{adjacency matrix} of a graph $G$ of order $n$ is the $n \times n$ matrix $\mathbf{A}_G$
whose $(i,j)$-entry is equal to $1$ whether $ij \in E(G)$ and is equal to $0$ otherwise.
This matrix $\mathbf{A}_G$ is symmetric and then it has $n$ real eigenvalues, herein indexed in non increasing order and denoted by
$\lambda_1(G) \ge \lambda_2(G) \ge \dots \ge \lambda_n(G)$. These eigenvalues of $\mathbf{A}_G$ are also called the \emph{eigenvalues} of $G$.
The \textit{spectrum} of $G$ is the multiset
$$
\sigma(G)=\{\mu_1^{[m_1]}, \mu_2^{[m_2]}, \ldots, \mu_q^{[m_q]}\},
$$
where $\mu_1, \mu_2, \ldots, \mu_q$ are the $q$ distinct eigenvalues of $G$ and $m_j$ denotes the multiplicity of $\mu_j$, for $j=1, 2, \ldots, q$. Therefore,
{\footnotesize$$
\underbrace{\lambda_1(G) = \cdots = \lambda_{m_1}(G)}_{\mu_1} > \underbrace{\lambda_{m_1+1}(G) = \cdots = \lambda_{m_1+m_2}(G)}_{\mu_2} > \cdots > \underbrace{\lambda_{n-m_q+1}(G) = \cdots = \lambda_{n}(G)}_{\mu_q}
$$}
and $\sum_{j=1}^{q}{m_j}=n$. The associated eigenspace of an eigenvalue $\lambda$ of a graph $G$ is denoted by ${\mathcal E}_G(\lambda)$.

The $(\kappa,\tau)$-regular sets were first investigated in \cite{thompson81}, in the context of regular graphs, where a subgraph of a graph
$G$ induced by a $(\kappa,\tau)$-regular set was called an \textit{eigengraph} of $G$. Furthermore, a necessary and sufficient condition for
the existence of a $(\kappa,\tau)$-regular set in a regular graph was deduced. Theorem~\ref{Th_Thompson} states a slightly different version
deduced in \cite{CardosoCvetkovic06} also for regular graphs. From now on, $\hat{\mathbf{e}}$ denotes the all-one vector with $n$ components
and the \textit{characteristic vector} of a vertex subset $S$ of a graph is $x(S) \in \{0,1\}^{V(G)}$ such that its $i$-th component is qual
to $1$ if $i \in V(G)$ and equal to $0$ otherwise.

\begin{theorem}\cite{thompson81,CardosoCvetkovic06}\label{Th_Thompson}
A $p$-regular graph $G$ has a $(\kappa,\tau)$-regular set, with $\kappa < p$, if and only if $\kappa - \tau$ is an eigenvalue of $G$ and
there exists $\mathbf{x} \in \{0,1\}^{V(G)}$ such that $\mathbf{x} - \frac{\tau}{p+\tau-\kappa}\hat{\mathbf{e}} \in {\mathcal E}_G(\kappa-\tau)$.
Furthermore, $\mathbf{x}$ is the characteristic vector of a $(\kappa-\tau)$-regular set.
\end{theorem}

More generally, the next theorem states a necessary and sufficient condition for the existence of a $(\kappa,\tau)$-regular set in an arbitrary
graph. This version is a stronger variant than the one presented in \cite{CardosoRama04}.

\begin{theorem}\label{kt_theorem}
A graph $G$ has a $(\kappa,\tau)$-regular set if and only if the linear system
\begin{equation}
\left(\mathbf{A}_G - (\kappa - \tau)I_n\right)x = \tau \hat{\mathbf{e}}, \label{kt_equation}
\end{equation}
where $I_n$ is the identity matrix of order $n$, has a $0-1$ solution. Furthermore, every $0-1$ solution is the characteristic vector of a
$(\kappa,\tau)$-regular set.
\end{theorem}

\begin{proof}
Let $S \subset V(G)$ be a $(\kappa,\tau)$-regular set. Then it is immediate that the characteristic vector of $S$, $x(S)$, is a solution of
\eqref{kt_equation}. Conversely, assuming that $\textbf{x}=(\textbf{x}_1, \ldots, \textbf{x}_n) \in \{0,1\}^{V(G)}$ is a solution of
\eqref{kt_equation}, its $i$-th equation, for $i=1, \ldots, n$, can be written as follows
\begin{eqnarray*}
\sum_{j \in N_G(i)}{\textbf{x}_j} - (\kappa - \tau)\textbf{x}_i = \tau & \Leftrightarrow & \sum_{j \in N_G(i): \textbf{x}_j=1}{\textbf{x}_j} = \left\{\begin{array}{ll}
                                                                                                          \kappa, & \hbox{if } \textbf{x}_i=1 \\
                                                                                                          \tau,   & \hbox{otherwise.}
                                                                                                          \end{array}\right.
\end{eqnarray*}
Therefore, the vertex set defined by $\textbf{x}$ as being its characteristic vector is $(\kappa,\tau)$-regular.
\end{proof}

Regarding the cardinality of $(\kappa,\tau)$-regular sets, the next theorem extends the result obtained in \cite{barbosa_cardoso04}
for $(0,\tau)$-regular sets. The graphs with $(0,\tau)$-regular sets were called $\tau$-regular-stable graphs in \cite{barbosa_cardoso04} .

\begin{theorem}\label{th_1.3}
Let $G$ be  graph with a $(\kappa,\tau)$-regular set $S \subset V(G)$, such that $\delta(G)+\tau > \kappa$. Then
\begin{equation}
\frac{n\tau}{\Delta(G) - (\kappa - \tau)} \le |S| \le \frac{n\tau}{\delta(G) - (\kappa - \tau)}.\label{cardinality}
\end{equation}
\end{theorem}

\begin{proof}
Since $S \subset V(G)$, from the definition of a $(\kappa,\tau)$-regular set, $\tau>0$ and we obtain
$$
|S|(\delta(G)-\kappa) \le (n-|S|)\tau \le |S|(\Delta(G)-\kappa).
$$
Therefore, $|S|\left(\delta(G)-(\kappa-\tau)\right) \le n\tau \le |S|\left(\Delta(G)- (\kappa-\tau)\right)$ and the inequalities
\eqref{cardinality} follow.
\end{proof}

As immediate consequence of Theorem~\ref{th_1.3}, if $G$ is a $p$-regular graph with a $(\kappa,\tau)$-regular set
$S \subset V(G)$, then $|S| = \frac{n\tau}{p - (\kappa - \tau)}$.

\section{Characterization of graphs with classical combinatorial structures}

There are several classical combinatorial structures which can be characterized by $(\kappa,\tau)$-regular sets as it is the case of perfect
matchings, Hamiltonian cycles, efficient dominating sets and dominating induced matchings (also called efficient edge dominating sets). Furthermore,
there are graphs with particular combinatorial structure that can be characterized ´by using $(\kappa,\tau)$-regular sets, as it is
the case of strongly regular graphs.

The next theorem which appear in \cite{barbosa_cardoso04} (see also \cite{cardoso01}) states a necessary and sufficient condition for the existence of perfect matchings in graphs.

\begin{theorem}\cite{barbosa_cardoso04}
A graph $G \ne K_2$ has a perfect matching if and only if its line graph $L(G)$ has a $(0,2)$-regular set.
\end{theorem}

It follows a necessary and sufficient condition for Hamiltonian graphs published in \cite{ACS2013} (for the reader convenience the proof is
also presented).

\begin{theorem}\cite{ACS2013}
A graph $G$ is Hamiltonian if and only if its line graph has a $(2,4)$-regular set inducing a connected subgraph.
\end{theorem}

\begin{proof}
Let $C$ be a Hamilton cycle of a graph $G$ and let $L(C)$ be the vertex subset of the line graph $L(G)$ corresponding to the edges in $C$.
Then it is immediate that $L(C)$ is a $(2,4)$-regular set of $L(G)$ inducing a connected subgraph.
Conversely, assume that the line graph $L(G)$ of a graph $G$ has a $(2,4)$-regular set $S$ inducing a connected subgraph. The edges of $G$
corresponding to $S$ form just one cycle $C$ in $G$. Furthermore, since each vertex not in $S$ has $4$ neighbors in $S$, the corresponding
edges in $G$ have both end-vertices in the cycle $C$. Therefore, $C$ is Hamiltonian.
\end{proof}

\begin{example}
The Figure~\ref{Hamiltonian} depicts a Hamiltonian graph $H$ and its line graph $L(H)$. The graph $H$ has a Hamiltonian cycle define by the
edge set $C=\{a, b, c, d, e, f, g\}$ and this edge set corresponds in $L(H)$ to a $(2,4)$-regular set inducing a connected subgraph.

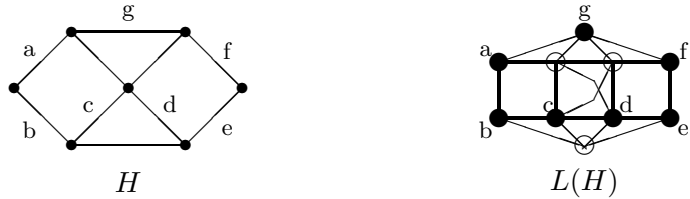
\begin{figure}[ht]
\begin{center}
\unitlength=0.5 mm
\begin{picture}(100,100)(10,70)
%\unitlength=0.7 mm
%\begin{picture}(150,100)(-70,50)
%H
\put(-25,135){\circle*{3}} % a
\put(-10,150){\circle*{3}} % b
\put(20,150){\circle*{3}} % c
\put(35,135){\circle*{3}} % d
\put(20,120){\circle*{3}} % e
\put(-10,120){\circle*{3}} % f
\put(5,135){\circle*{3}} % g
\put(-10,150){\line(-1,-1){15}} %(e2,e1)
\put(-21,145){\makebox(0,0){\footnotesize a}}
\put(-10,120){\line(-1,1){15}} %(e6,e1)
\put(-21,124){\makebox(0,0){\footnotesize b}}
\put(5,135){\line(-1,-1){15}}
\put(-5.5,130){\makebox(0,0){\footnotesize c}}
\put(5,135){\line(1,-1){15}}
\put(16.0,131){\makebox(0,0){\footnotesize d}}
\put(20,120){\line(1,1){15}} %(e5,e4)
\put(31.0,124){\makebox(0,0){\footnotesize e}}
\put(20,150){\line(1,-1){15}} %(e3,e4)
\put(31.0,145){\makebox(0,0){\footnotesize f}}
\put(-10,150){\line(1,0){30}} %(e2,e3)
\put(5,155){\makebox(0,0){\footnotesize g}}
\put(-10,120){\line(1,0){30}} %(e6,e5)
\put(5,135){\line(1,1){15}}
\put(5,135){\line(-1,1){15}}
\put(5,110){\makebox(0,0){$H$}}
%
%L(H)
%
\put(102.5,142){\circle*{5}}
\put(99,145){\makebox(0,0){\footnotesize a}}
\put(102.5,127){\circle*{5}}
\put(99,124){\makebox(0,0){\footnotesize b}}
\put(117.5,142){\circle{5}}
\put(117.5,142){\line(2,-1){10}}
\put(117.5,127){\circle*{5}}
\put(115.5,130){\makebox(0,0){\footnotesize c}}
\put(117.5,127){\line(2,1){10}}
\put(132.5,142){\circle{5}}
\put(132.5,142){\line(-1,-2){5}}
\put(132.5,127){\circle*{5}}
\put(136.0,131){\makebox(0,0){\footnotesize d}}
\put(132.5,127){\line(-1,2){5}}
\put(147.5,127){\circle*{5}}
\put(151.0,124){\makebox(0,0){\footnotesize e}}
\put(147.5,142){\circle*{5}}
\put(151.0,145){\makebox(0,0){\footnotesize f}}
\put(125,150){\circle*{5}}
\put(125,155){\makebox(0,0){\footnotesize g}}
\put(125,120){\circle{5}}
\linethickness{1.0pt}
\put(102.5,142){\line(1,0){45}}
\put(102.5,127){\line(1,0){45}}
\put(102.5,127){\line(0,1){15}}
\put(117.5,127){\line(0,1){15}}
\put(132.5,127){\line(0,1){15}}
\put(147.5,127){\line(0,1){15}}
\put(102.5,142){\line(3,1){23}}
\put(147.5,142){\line(-3,1){23}}
\put(102.5,127){\line(3,-1){23}}
\put(147.5,127){\line(-3,-1){23}}
\put(117.5,127){\line(1,-1){8}}
\put(117.5,142){\line(1,1){8}}
\put(132.5,142){\line(-1,1){8}}
\put(132.5,127){\line(-1,-1){8}}
\put(125,110){\makebox(0,0){$L(H)$}}
\end{picture}
\caption{A Hamiltonian graph $H$ and its line graph $L(H)$ with the $(2,4)$-regular set $S=\{a, b, c, d, e, f, g\}$.}\label{Hamiltonian}
\end{center}
\end{figure}
\end{example}

In \cite{ScirihaCardoso12} it was proved that a graph $G$ is Hamiltonian if and only the subdivision of $G$ (that is, a graph obtained from
$G$ after inserting a vertex in the middle of each edge) has a $(2,2)$-regular set inducing a connected subgraph.\\

Before to proceed, it is worth recall some domination concepts. Given a graph $G$ and a vertex $v \in V(G)$, $v$ \textit{dominates} itself
and all its neighbors. A vertex set $S \subset V(G)$ is \textit{dominating} if every vertex of $G$ is dominated by at least one vertex of $S$.
The \textit{domination number} of a graph $G$, $\gamma(G)$, is the cardinality of a dominating set in $G$ with minimum cardinality.
A \textit{dominating set} $S$ is \textit{efficient dominating} (also known as \textit{independent perfect dominating set}) if each vertex of
$G$ is dominated by precisely one vertex of $S$. Not every graph has an efficient dominating set (for example, $C_4$ has no efficient dominating
sets). The problem of determining an efficient dominating set in a graph (if there exists) is called the \textit{efficient dominating set
problem}. This problem is known to be ${\cal NP}$-complete for general graphs \cite{BBS88}. A closely related problem is that of determining
if $G$ has an \textit{efficient edge dominating} set, that is, a set $E$ of edges such that every edge of $G$ shares a vertex with precisely
one edge in $E$ (assuming that an edge shares a vertex with itself). This edge set is also known as a \textit{dominating induced matching problem}
and the problem of determining such edge set is also ${\cal NP}$-complete \cite{GSSH93}. An instance of a dominating induced matching can be
transformed into an instance of an efficient dominating set by associating to the input graph $G$ its line graph $L(G)$.

\begin{theorem}
A vertex subset $S$ of a graph $G$ is an efficient dominating set if and only if $S$ is $(0,1)$-regular.
\end{theorem}

\begin{proof}
Taking into account the definitions of a $(\kappa,\tau)$-regular set and efficient dominating set, the result follows.
\end{proof}

The $(\kappa,\tau)$-regular sets are also related with determination of vertex subsets of maximum cardinality inducing a $\kappa$-regular subgraph,
as it is highlighted by the next theorem.

\begin{theorem}\cite{CKL07}\label{Th_CKL}
Let $G$ be a graph of order $n$ and let $\tau=-\lambda_{n}(G)$. If $S \subset V(G)$ is $(\kappa,\kappa+\tau)$-regular, then $S$ is a maximum cardinality vertex
subset of $G$ inducing a $\kappa$-regular subgraph.
\end{theorem}

\begin{example}
Consider the graph $G$ depicted in Figure~\ref{Figure_ckl} for which $\sigma(G)=\{-2^{[2]},0^{[3]},4^{[1]}\}$ and the vertex subset $S=\{1,3,4,6\}$ is $(2,4)$-regular.
Therefore, applying Theorem~\ref{Th_CKL}, we may conclude that $S$ is a maximum cardinality vertex subset inducing a $2$-regular subgraph.
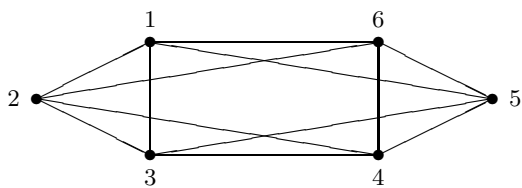
\begin{figure}[ht]
\begin{center}
\unitlength=0.3 mm
\begin{picture}(200,80)(150,70)
%
%Arestas do grafo
\put(200,135){\line(1,0){100}} %1-6
\put(200,85){\line(1,0){100}}  %3-4
\put(150,110){\line(2,1){50}}  %2-1
\put(150,110){\line(2,-1){50}} %2-3
\put(200,85){\line(0,1){50}} %5-1
\put(300,135){\line(2,-1){50}} %6-5
\put(300,85){\line(0,1){50}} %4-6
\put(150,110){\line(6,1){150}} %2-6
\put(150,110){\line(6,-1){150}} %2-4
\put(350,110){\line(-6,1){150}} %5-1
\put(350,110){\line(-6,-1){150}} %5-4
\put(300,85){\line(2,1){50}}%3-5
%
%Representacao dos vertices
\put(150,110){\circle*{4.7}} % {2}
\put(350,110){\circle*{4.7}} % {3}
\put(200,135){\circle*{4.7}} % {1}
\put(300,135){\circle*{4.7}} % {6}
\put(200,85){\circle*{4.7}} % {5}
\put(300,85){\circle*{4.7}} % {4}
%
%Etiquetacao dos vertices de Gamma1
%
\put(140,110){\makebox(0,0){\footnotesize 2}}
\put(360,110){\makebox(0,0){\footnotesize 5}}
\put(200,145){\makebox(0,0){\footnotesize 1}}
\put(300,145){\makebox(0,0){\footnotesize 6}}
\put(200,75){\makebox(0,0){\footnotesize 3}}
\put(300,75){\makebox(0,0){\footnotesize 4}}
\end{picture}
\caption{Graph $G$ with the $(2,4)$-regular set $S=\{1,3,4,6\}$ and $\sigma(G)=\{-2^{[2]},0^{[3]},4^{[1]}\}$.}\label{Figure_ckl}
\end{center}
\end{figure}
\end{example}
A \textit{strongly regular graph} with parameters $(n,p,a,c)$ is a $p$-regular graph of order $n$, where each pair of vertices have $a$ common
neighbors if they are adjacent and $c$ common neighbors otherwise. For instance, the Petersen graph depicted in Figure~\ref{PetersenGraph} is a
strongly regular graph with parameters $(10,3,0,1)$ and the graph depicted in Figure~\ref{Figure_ckl} is a strongly regular graph with parameters
$(6,4,2,4)$. A strongly regular graph $G$ is \textit{primitive} if $G$ and the complement graph of $G$
are both connected; otherwise it is called \textit{imprimitive}. A strongly regular graph with parameters $(n,p,a,c)$ is imprimive if and only
if $c=p$ or $c=0$ (see \cite[p. 178]{godsil93}. The graph depicted in Figure~\ref{Figure_ckl} is an example of an imprimitive strongly regular
graph.\\

As was noticed in \cite{CardosoDelormeRama07}, if $S$ is a maximum stable set of a primitive strongly regular graph $G$ with parameters $(n,p,a,c)$,
then  for all $k$ in $S$
$$
|S| = \sum_{j \in N_G(k)}{\frac{|N_G(j) \cap S|}{c}} - \frac{p-c}{c}.
$$
Furthermore, if $S$ is $(0,\tau)$-regular, then
$$
|S| = \frac{p(\tau-1)+c}{c}.
$$

The next theorem gives a necessary and sufficient condition for a regular graph be strongly regular, based on the existence of $(\kappa,\tau)$-regular
sets.

\begin{theorem}\cite{CSZ2010}\label{Th-srg}
A $p$-regular graph $G$ is strongly regular with parameters $(n,p,a,c)$ if and only if for every vertex $v \in V(G)$, the
vertex subset $S=N_G(v)$ is $(a,c)$-regular in $G-v$ (the graph obtained from $G$ after the deletion of the vertex $v$).
\end{theorem}

\begin{example}
Consider the strongly regular graph $G$ with parameters $(10,3,0,1)$ depicted in Figure~\ref{PetersenGraph}. Deleting the vertex $1$
(for instance), the vertex subset $N_G(1)=\{5, 7, 9\}$ is $(0,1)$-regular in the graph $G-1$. A similar $(0,1)$-regular set is
determined deleting any other vertex of $G$.
\end{example}

\section{Determination of $(\kappa,\tau)$-regular sets}

As a consequence of the results obtained in \cite{halldorsson_et_al00}, in general, the recognition of graphs with a $(\kappa,\tau)$-regular
set is ${\cal NP}$-complete. However, there are families of graphs for which such recognition and the determination of $(\kappa,\tau)$-regular
sets can be done in polynomial time, in particular for the graphs whose maximum multiplicity of the eigenvalues is small. This section is
devoted to the results and algorithmic techniques developed for the determination of $(\kappa,\tau)$-regular sets.

The next theorem is a variant of a theorem which appears in \cite{CLLP2016}.

\begin{theorem}
Let $G$ be a graph with at least one edge and let $\overline{\mathbf{x}}$ be a particular solution of the linear system \eqref{kt_equation}.
The graph $G$ has a $(\kappa,\tau)$-regular set if and only if there is a $0-1$ vector $\mathbf{x}$ such that
\begin{equation}
\mathbf{x} = \overline{\mathbf{x}} + \hat{\textbf{u}}, \label{kt_solution}
\end{equation}
where $\hat{\textbf{u}}=\textbf{0}$ if $\kappa-\tau$ is not an eigenvalue of $G$ and $\hat{\textbf{u}} \in {\mathcal E}_G(\kappa-\tau)$ otherwise.
Furthermore, every $0-1$ solution $\mathbf{x}$ in \eqref{kt_solution} is the characteristic vector of a $(\kappa,\tau)$-regular set $S \subset V(G)$.
\end{theorem}

\begin{proof}
It is immediate that every solution of the linear system \eqref{kt_equation} can be obtained from the equation \eqref{kt_solution}.
Therefore, applying Theorem~\ref{kt_theorem}, the result follows.
\end{proof}

\begin{corollary}\cite{CLLP2016}\label{kt_size}
If a graph $G$ has a $(\kappa,\tau)$-regular set $S \subseteq V(G)$ and $\overline{\mathbf{x}}$ is a particular solution of the linear system
\eqref{kt_equation}, then $|S| = \hat{\mathbf{e}}^T\overline{\mathbf{x}}$.
\end{corollary}

Note that the determination of a $(\kappa,\tau)$-regular set of a graph $G$ (if there exists) or the conclusion that such vertex subset does
not exists is easy to check, when $\kappa - \tau$ is not an eigenvalue of $G$. Indeed, in such a case the linear system \eqref{kt_equation}
has an unique solution which if it is $0-1$, is the characteristic vector of the unique $(\kappa,\tau)$-regular set of $G$;
otherwise there is no $(\kappa,\tau)$-regular set in $G$. Furthermore, from Corollary~\ref{kt_size}, considering a particular solution of the
linear system \eqref{kt_equation}, $\overline{\mathbf{x}}$, if $\hat{e}^T\overline{\mathbf{x}}$ is not a positive integer, then $G$ has no
$(\kappa,\tau)$-regular sets.

\begin{theorem}\label{kt_determination}\cite{CLP2018}
Let $G$ be a graph with a $(\kappa,\tau)$-regular set $S \subset V(G)$ and let $\overline{\mathbf{x}}$ be a particular solution of the
linear system \eqref{kt_equation}. Assuming that $\kappa-\tau$ is an eigenvalue of $G$ with multiplicity $t$, then the characteristic vector
of $S$, $\mathbf{x}$, is determined by the equality
\begin{equation}
\mathbf{x} = \overline{\mathbf{x}} + \sum_{j=1}^t{\delta_{i_j} \hat{\mathbf{v}}_j},\label{main_equality}
\end{equation}
where $\delta_{i_j} \in \{-\overline{\mathbf{x}}_{i_j}, 1-\overline{\mathbf{x}}_{i_j}\}$, for $j=1, \dots, t$, the vectors
$\hat{\mathbf{v}}_1, \dots, \hat{\mathbf{v}}_t \in {\mathcal E}_G(\kappa - \tau)$ and are such that the matrix $V$ whose columns are
$\hat{\mathbf{v}}_1, \dots, \hat{\mathbf{v}}_t$ has a $t \times t$ identity submatrix defined by the $t$ rows of $V$ with indices
in $I=\{i_1, i_2, \ldots, i_t\} \subset V(G)$.
\end{theorem}

\begin{proof}
Let $U$ be a matrix whose columns are $t$ linear independent eigenvectors, $\hat{\mathbf{u}}_1, \dots, \hat{\mathbf{u}}_t$, associated to the
eigenvalue $\kappa-\tau$. Since these vectors are linear independent, the matrix $U$ has a subset of $t$ rows indexed by $i_1, \dots, i_t$,
defining a nonsingular $t \times t$ submatrix $T$. Then, we can replace each column of $U$ by a linear combination of columns of $U$ to obtain
a matrix $V$ whose columns $\hat{\mathbf{v}}_1, \dots, \hat{\mathbf{v}}_t$ remain as linear independent eigenvectors associated to $\kappa-\tau$
and the corresponding $t \times t$ submatrix $T'$ of $V$ defines the identity matrix of order $t$. Therefore, considering that $\mathbf{x}$ is the
characteristic vector of $S$, from the system \eqref{kt_solution}, where $\mathbf{u}$ is replaced by $\sum_{j=1}^t{\delta_{i_j}\textbf{v}_j}$,
it follows that
\begin{eqnarray*}
\mathbf{x}_{i_1} = \overline{\mathbf{x}}_{i_1} + \delta_{i_1} & \Leftrightarrow & \delta_{i_1} = \mathbf{x}_{i_1} - \overline{\mathbf{x}}_{i_1} \\
\mathbf{x}_{i_2} = \overline{\mathbf{x}}_{i_2} + \delta_{i_2} & \Leftrightarrow & \delta_{i_2} = \mathbf{x}_{i_2} - \overline{\mathbf{x}}_{i_2} \\
\vdots \qquad \qquad \;                                                & \vdots          & \qquad \vdots \\\
\mathbf{x}_{i_t} = \overline{\mathbf{x}}_{i_t} + \delta_{i_t} & \Leftrightarrow & \delta_{i_t} = \mathbf{x}_{i_t} - \overline{\mathbf{x}}_{i_t}.
\end{eqnarray*}
Since $\mathbf{x}_{i_j} \in \{0,1\}$, for $j=1, \dots, t$, the result follows.
\end{proof}

As immediate consequence of Theorem~\ref{kt_determination}, the Algorithm 1 decides in a finite number of steps whether or not a graph $G$, having
an eigenvalue $\kappa-\tau$ with multiplicity $t$, has a $(\kappa,\tau)$-regular set, determining such vertex subset when it exists.

% Algoritmo 1
\begin{algorithm}[H]%
\begin{algorithmic}[1]
\REQUIRE{$\mathbf{A}_G$, $\kappa$,  $\tau$, $t$ and the $n \times t$ matrix $U$ whose columns are linear
         independent vectors of $\mathcal{E}_{G}(\kappa-\tau)$ when $t>0$.}
\ENSURE{a $(\kappa,\tau)$-regular set $S$ or the conclusion that it does not exists.}
\STATE \textbf{COMPUTE} a particular solution $\overline{\mathbf{x}}$ of the linear system \eqref{kt_equation}.
\STATE \textbf{SET} $\mathbf{x} := \overline{\mathbf{x}}$.
\STATE \textbf{IF} $t=0$, \textbf{THEN} \textbf{GOTO} Step~\ref{xyz}
\STATE \textbf{DETERMINE} the matrix $V$, with columns $\hat{\mathbf{v}}_1, \ldots, \hat{\mathbf{v}}_t$, obtained
       from $U$ as in Theorem~\ref{kt_determination} and the index subset $I=\{i_1, \dots i_t\} \subset \{1, \dots, n\}$.\label{abc}
\STATE \textbf{SET} $\Lambda := \{(\delta_{i_1}, \ldots, \delta_{i_t}): \delta_{i_j} \in \{-\overline{\mathbf{x}}_{i_j},1-\overline{\mathbf{x}}_{i_j}\}, i_j \in I\}$.\label{uvw}
\STATE \textbf{WHILE} $\mathbf{x}$ is not $0-1$ and $\Lambda \ne \emptyset$
\STATE \qquad \textbf{CHOOSE} $(\delta_{i_1}, \ldots, \delta_{i_t}) \in \Lambda$ and \textbf{SET} $\Lambda := \Lambda \setminus \{(\delta_{i_1}, \ldots, \delta_{i_t})\}$.
\STATE \qquad \textbf{IF} $\overline{\mathbf{x}} + \sum_{j=1}^{t}{\delta_{i_j}\hat{\mathbf{v}}_j}$ is a $0-1$ vector \textbf{THEN}
               \textbf{SET} $\mathbf{x} := \overline{\mathbf{x}} + \sum_{j=1}^{t}{\delta_{i_j}\hat{\mathbf{v}}_j}$.
\STATE \qquad \textbf{END If}
\STATE \textbf{END WHILE}
\STATE \textbf{IF} $\mathbf{x}$ is $0-1$ \textbf{THEN} return $\mathbf{x}$ as the characteristic vector of $S$.\label{xyz}
\STATE \qquad \qquad \qquad \textbf{ELSE} $G$ has no $(\kappa,\tau)$-regular sets.
\STATE \textbf{END IF}\\
\end{algorithmic}%
\caption{for the recognition whether a graph $G$ has a $(\kappa,\tau)$ regular set and its determination when it exists.\label{algorithm}}%
\end{algorithm}%

\begin{example}
Let us apply Algorithm 1 repeatedly (updating $\Lambda$ in each run of step~\ref{uvw}, removing the t-tuples already determined which are
0-1 solutions) to the determination of all $(0,2)$-regular sets $S_1$ and all $(1,3)$-regular sets $S_2$ in the graph $G$ depicted in
Figure~\ref{PetersenGraph}. Note that the adjacency matrix $\textbf{A}_G$ has three distinct eigenvalues: $3$ with multiplicity $1$, $1$
with multiplicity $5$ and $-2$ with multiplicity $4$. Since $0-2=1-3=-2$, in both cases we can consider the matrix
$$
U^T = \left(\begin{array}{rrrrrrrrrr}
                   -1 &-1 &-1 & 0 & 1 & 0 & 1 & 0 & 0 & 1 \\
                   -3 &-1 &-1 &-1 & 2 & 0 & 2 & 0 & 2 & 0 \\
                    1 &-1 & 1 &-1 & 0 & 0 &-2 & 2 & 0 & 0 \\
                    1 & 1 &-1 &-1 &-2 & 2 & 0 & 0 & 0 & 0 \\
            \end{array}\right),
$$
where the rows of $U^T$ are linear independent eigenvectors belonging to ${\mathcal E}_G(\kappa-\tau)$. The matrix $V$ obtained from
the matrix $U$ in step~\ref{abc} can take the form
$$
V = \left(\begin{array}{rrrr}
           1 & 0 & 0 & 0 \\
           0 & 1 & 0 & 0 \\
           0 & 0 & 1 & 0 \\
           0 & 0 & 0 & 1 \\
           -2/3 &-2/3 & 1/3 & 1/3 \\
            1/3 & 1/3 &-2/3 &-2/3 \\
           -2/3 & 1/3 &-2/3 & 1/3 \\
            1/3 &-2/3 & 1/3 &-2/3 \\
           -2/3 & 1/3 & 1/3 &-2/3 \\
            1/3 &-2/3 &-2/3 & 1/3
          \end{array}\right)
$$
and then $I=\{1,2,3,4\}$. For each case, consider a particular solution $\overline{\textbf{x}}$ of the linear system \eqref{kt_equation}.
\begin{enumerate}
\item For $(\kappa,\tau)=(0,2)$, using the particular solution of \eqref{kt_equation} $\overline{\textbf{x}}=\frac{2}{5}\hat{e}$,
      by Corollary~\ref{kt_size} $|S_1|=4$ and by Theorem~\ref{kt_determination} $\delta_1, \delta_2,\delta_3,\delta_4 \in \{-\frac{2}{5},\frac{3}{5}\}$.
      The $0-1$ solutions $\mathbf{x}$ in \eqref{main_equality} (and consequently the characteristic vectors of the $(0,2)$-regular sets)
      are obtained for each tuple $(\delta_1, \delta_2, \delta_3, \delta_4)$ of the table (in each row, $\delta_1, \dots, \delta_4$ appear
      in the first $4$ entries and the corresponding vertices of $S_1$ appear in the last $4$ entries):
      {\begin{center}\footnotesize
       \begin{tabular}{|r|r|r|r|cccc|}\hline
          $\delta_1$   & $\delta_2$   &  $\delta_3$  & $\delta_4$  &   &   &   &   \\ \hline
         $-2/5$ & $-2/5$ & $-2/5$ & $3/5$  & 4 & 5 & 7 & 10\\
         $-2/5$ & $-2/5$ & $3/5$  & $-2/5$ & 3 & 5 & 8 & 9 \\
         $-2/5$ & $3/5$  & $-2/5$ & $3/5$  & 2 & 6 & 7 & 9 \\
         $3/5$  & $-2/5$ & $-2/5$ & $-2/5$ & 1 & 6 & 8 & 10\\
         $3/5$  & $3/5$  & $3/5$  & $3/5$  & 1 & 2 & 3 & 4 \\ \hline
       \end{tabular}
      \end{center}}
\item For $(\kappa,\tau)=(1,3)$, using the particular solution of \eqref{kt_equation} $\overline{\textbf{x}}^T=(3/2,3/2,3/2,3/2,0,0,0,0,0,0)$,
      by Corollary~\ref{kt_size} $|S_2|=6$ and by Theorem~\ref{kt_determination}  $\delta_1,\delta_2,\delta_3,\delta_4 \in \{-\frac{3}{2}, -\frac{1}{2}\}$.
      The $0-1$ solutions $\mathbf{x}$ in \eqref{main_equality} (and consequently the characteristic vectors of the $(1,3)$-regular sets) are
      obtained for the tuples $(\delta_1, \delta_2, \delta_3, \delta_4)$ of the table:
      {\begin{center}\footnotesize
       \begin{tabular}{|r|r|r|r|cccccc|}\hline
         $\delta_1$ & $\delta_2$ &  $\delta_3$ & $\delta_4$ &   &   &   &   &   &   \\ \hline
         $-3/2$     & $-3/2$     & $-3/2$      & $-3/2$     & 5 & 6 & 7 & 8 & 9 & 10\\
         $-3/2$     & $-1/2$     & $-1/2$      & $-1/2$     & 2 & 3 & 4 & 5 & 7 & 9 \\
         $-1/2$     & $-3/2$     & $-1/2$      & $-1/2$     & 1 & 3 & 4 & 5 & 8 & 10\\
         $-1/2$     & $-1/2$     & $-1/2$      & $-3/2$     & 1 & 2 & 3 & 6 & 8 & 9 \\
         $-1/2$     & $-1/2$     & $-3/2$      & $-1/2$     & 1 & 2 & 4 & 6 & 7 & 10 \\ \hline
       \end{tabular}
      \end{center}}
\end{enumerate}
Note that in both cases, among the $2^4=16$ possible tuples $(\delta_1, \delta_2, \delta_3, \delta_4)$, we found $5$ tuples
producing $0-1$ solutions.
\end{example}

Note that in Theorem~\ref{kt_determination}, despite the set of possible tuples $(\delta_{i_1}, \delta_{i_2}, \ldots, \delta_{i_t})$ is finite,
its cardinality is the exponential number $2^t$. Therefore, when the multiplicity $t$ of the eigenvalue $\kappa-\tau$ is larger, the determination
of a tuple of scalers $(\delta_{i_1}, \delta_{i_2}, \ldots, \delta_{i_t})$ producing a $0-1$ solution in \eqref{main_equality} or the recognition
that there is no such solution, can not be computationally effective by checking all possible tuples. Considering the inequality (see, for instance,
\cite{CRS10})
$$
\max_{\mu \in \sigma(G) \setminus \{0\}} m_G(\mu) \le n - \gamma(G),
$$
where $m_G(\mu)$ is the multiplicity of the eigenvalue $\mu$ of the graph $G$, it follows that the graphs with higher domination number have
smaller upper bound for the maximum multiplicity of the nonzero eigenvalues of $G$. Thus, for those graphs the determination of $(\kappa,\tau)$-regular
sets with $\kappa-\tau \ne 0$ or the recognition that none of them there exist is computationally effective. Anyway, with the same purpose, the development of
computationally effective techniques for particular graph families remains an open problem and it is an interesting research line.

\section{Spectral properties}

The presence of $(\kappa,\tau)$-regular sets in graphs has deep influence on their spectrum. From Theorem~\ref{Th_Thompson}
it is immediate that a regular graph with a $(\kappa,\tau)$-regular set has $\kappa-\tau$ as an eigenvalue. However, in the
case of non-regular graphs, the presence of a $(\kappa,\tau)$-regular set does not imply that $\kappa-\tau$ is an  eigenvalue.
For instance, the graph $G$ depicted in Figure~\ref{fig4} has the $(1,1)$-regular set $\{b, e\}$ but $0$ is not an eigenvalue of $G$.
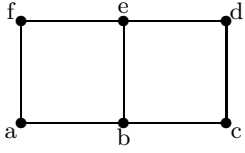
\begin{figure}[h]
\begin{center}
\unitlength=0.45 mm
\begin{picture}(205,45)(-200,113)
% vertices
\put(-140,120){\circle*{3}} % a
\put(-143,117){\makebox(0,0){\footnotesize a}}
\put(-110,120){\circle*{3}}  % b
\put(-110,116){\makebox(0,0){\footnotesize b}}
\put(-80,120){\circle*{3}} % c
\put(-77,117){\makebox(0,0){\footnotesize c}}
\put(-80,150){\circle*{3}} % d
\put(-77,153){\makebox(0,0){\footnotesize d}}
\put(-110,150){\circle*{3}}% e
\put(-110,154){\makebox(0,0){\footnotesize e}}
\put(-140,150){\circle*{3}} % f
\put(-143,153){\makebox(0,0){\footnotesize f}}
%
% edges
\put(-140,120){\line(0,1){30}} %(1,6)
\put(-110,120){\line(0,1){30}} %(2,5)
\put(-80,120){\line(0,1){30}} %(3,4)
%horizontal
\put(-140,120){\line(1,0){60}} %(1,2,3)
\put(-140,150){\line(1,0){60}} %(6,5,4)
\end{picture}
\caption{A graph $G$ with a $(1,1)$-regular set and spectrum $\sigma(G)=\{-1-\sqrt{2}, -1, 1-\sqrt{2}, -1+\sqrt{2}, 1, 1+\sqrt{2}\}$}\label{fig4}
\end{center}
\end{figure}
The graph $G$ of Figure~\ref{fig4} has also the $(2,1)$-regular sets $\{a,b,e,f\}$ and $\{b,c,d,e\}$, the $(1,2)$-regular set $\{a,c,d,f\}$ and
the $(0,1)$-regular sets $\{a, d\}$ and $\{c, f\}$.\\

The next theorem establishes a sufficient condition for arbitrary graphs with $(\kappa,\tau)$-regular sets to have $\kappa-\tau$ as an eigenvalue.

\begin{theorem}\cite{CardosoRama07}\label{Th_CR07}
Let $\lambda$ be an integer and let $G$ be a graph with a $(\kappa_1,\tau_1)$-regular set $S_1$, with $\tau_1>0$, and a $(\kappa_2,\tau_2)$-regular
set $S_2$, such that $S_1 \ne S_2$ and $k_1 - \tau_1 = \kappa_2 - \tau_2 = \lambda$. Then $\lambda \in \sigma(G)$ and
$\hat{\mathbf{u}} = \frac{\tau_2}{\tau_1}\mathbf{x} - \mathbf{y} \in \mathcal{E}_{G}(\lambda)$, where $\mathbf{x}$ is
the characteristic vector of $S_1$ and $\mathbf{y}$ is the characteristic vector of $S_2$.
\end{theorem}

\begin{example}
Consider the graph $G$ depicted in Figure~\ref{fig4}. Since the vertex subsets $S_1=\{a, b, e, f\}$ and $S_2=\{b, c, d, e\}$ are both $(2,1)$-regular
and the vertex subsets $T_1=\{a, d\}$ and $T_2=\{c, f\}$ are both $(0,1)$-regular, applying Theorem~\ref{Th_CR07}, it follows that $1$ and $-1$ are
both eigenvalues of $G$. Furthermore,
\begin{eqnarray*}
\bordermatrix{  &$\;$ \cr
              a & 1 \cr
              b & 0 \cr
              c &-1 \cr
              d &-1 \cr
              e & 0 \cr
              f & 1 \cr } \in {\mathcal E}_G(1) & \text{ and } & \bordermatrix{  &$\;$ \cr
                                                                               a & 1 \cr
                                                                               b & 0 \cr
                                                                               c &-1 \cr
                                                                               d & 1 \cr
                                                                               e & 0 \cr
                                                                               f &-1 \cr } \in {\mathcal E}_G(-1).
\end{eqnarray*}
\end{example}

Now, it is worth mention the concepts of main and non-main eigenvalues. An eigenvalue $\mu$ of a graph $G$, which has an associated eigenspace
${\mathcal E}_G(\mu)$ not orthogonal to the all-one vector $\hat{\mathbf{e}}$, is said to be \textit{main}. When ${\mathcal E}_G(\mu)$ is orthogonal to
$\hat{\mathbf{e}}$ the eigenvalue $\mu$ is referred as {\it non-main}. The concept of main (non-main) eigenvalue was introduced in 1970 by Cvetkovi\'{c}
\cite{Cvetkovic1970} (see also \cite{CDS79,CRS97}) and further investigated in many papers since then. For every graph $G$, its largest eigenvalue
$\lambda_1(G)$ is a main eigenvalue. In particular, it is well known that a graph  is regular if and only if it has only one main eigenvalue (see
\cite{Rowlinson07}, where a survey on main (non-main) eigenvalues was published).

Regarding graphs with just two main eigenvalues we may use $(\kappa ,\tau)$-regular sets for the determination of particular families using a graph
operation introduced in \cite{CSZ2010} and herein described as follows:

Consider the graph operation $H = G_1 \bigoplus_s^{\tau} G_2$, where $G_1$ is a ${\kappa }_1$-regular graph, $G_2$ is a ${\kappa }_2$-regular graph
and $H$ is obtained from $G_1$ and $G_2$ by connecting each vertex of $V(G_1)=\{x_1, \ldots, x_{n_1}\}$ to $\tau > 0$ vertices in
$V(G_2)=\{y_1, \ldots, y_{n_2}\}$, such that $B_i=N_H(x_i) \cap V(G_2), i=1, \ldots, n_1$, is a $1-(n_2,\tau,s)$ combinatorial
design \cite{Stinson03} (that is, $\bigcup_{i=1}^{n_1}{B_i} = V(G)$, $|B_i|=\tau$, for $i=1, \ldots, n_1$ and each vertex $v \in V(G_2)$ belongs
to exactly $s$ blocks $B_i$). Therefore, the vertex subsets $V(G_1)$ and $V(G_2)$ are $(\kappa_1,s)$-regular and $(\kappa_2,\tau)$-regular,
respectively.

\begin{theorem}\cite{CSZ2010}\label{ThCSZ}
Considering a ${\kappa}_1$-regular graph $G_1$ and a ${\kappa}_2$-regular graph $G_2$, let $H = G_1 \bigoplus_s^{\tau} G_2$ be
the graph obtained as above described. If $\mu$ is a main eigenvalue of $H$, then
\begin{equation} \label{2main}
\mu = \frac{{\kappa }_1+{\kappa }_2 \pm \sqrt{({\kappa }_2 - {\kappa }_1)^2 + 4s\tau}}{2}.
\end{equation}
\end{theorem}

The particular case of unicyclic graphs with just two main eigenvalues was investigated in \cite{HouTian06}.

\begin{theorem}\cite{HouTian06}\label{ThHouTian}
The graphs attained from a cycle $C_n$ by attaching $s > 0$ pendent vertices are all unicyclic graphs with exactly two main eigenvalues.
\end{theorem}

Note that any unycliclic graph $H$ as referred in Theorem~\ref{ThHouTian} is such that $H=snK_1 \bigoplus_s^{1} C_n$ and by Theorem~\ref{ThCSZ}
its main eigeinvalues are
$$
\mu = 1 \pm \sqrt{1+s}.
$$
Since the largest eigenvalue of a graph is main, it follows that $\mu=1+\sqrt{1+s}$ is the largest eigenvalue of $H$.

\bigskip

\begin{theorem}\cite{CSZ2010}\label{Th_CSZ}
Let $S$ be a $(\kappa,\tau)$-regular set of a graph $G$, with characteristic vector $\mathbf{x}(S)$ and let $\mu$ be an eigenvalue of $G$. Then $\mu$
is non-main if and only if
$$
\mu = \kappa - \tau  \qquad \text{ or } \qquad \mathbf{x}(S) \in \left ({\mathcal E}_G(\mu)\right )^{\bot},
$$
where $\left({\mathcal E}_G(\mu)\right)^{\bot}$ denotes the vector space orthogonal to ${\mathcal E}_G(\mu)$.
\end{theorem}

From the above theorem we have the following corollary.

\begin{corollary}\cite{CSZ2010}
Let $\kappa, \tau \in {\mathbb Z}^+\cup\{0\}$, with $\tau >0,$ where ${\mathbb Z}^+$ denotes the set of positive integers. If $\mu = \kappa - \tau$
is a main eigenvalue of $G$, then $G$ does not have a $(\kappa,\tau)$-regular set.
\end{corollary}

Now we present the following theorem.

\begin{theorem}
Let $G$ be a graph with a $(\kappa,\tau)$-regular set $S \subset V(G)$ and let $\overline{\mathbf{x}}$ be a particular solution of \eqref{kt_equation}.
If $\mu$ is a main eigenvalue of $G$ and $\hat{\mathbf{u}} \in {\mathcal E}_G(\mu)$ is not orthogonal to $\hat{\mathbf{e}}$, then
\begin{equation}
\hat{\mathbf{u}}^T\overline{\mathbf{x}} \ne 0 \qquad \text{ and } \qquad \mu = \tau \frac{\hat{\mathbf{u}}^T\hat{\mathbf{e}}}{\hat{\mathbf{u}}^T\overline{\mathbf{x}}}+ (\kappa-\tau). \label{marca1}
\end{equation}
\end{theorem}

\begin{proof}
Since $\mu$ is a main eigenvalue of $G$, by Theorem~\ref{Th_CSZ}, $\mu \ne k-\tau$.  Taking into account that $\overline{\mathbf{x}}$ is a particular solution of \eqref{kt_equation},
multiplying both sides of \eqref{kt_equation} on the left by $\hat{\mathbf{u}}^T$, we obtain
\begin{equation}
(\mu - (\kappa-\tau))\hat{\mathbf{u}}^T\overline{\mathbf{x}} = \tau \hat{\mathbf{u}}^T\hat{\mathbf{e}}. \label{marca2}
\end{equation}
Therefore, $\hat{\mathbf{u}}^T\overline{\mathbf{x}} \ne 0$ and from \eqref{marca2} it follows
$\mu = \tau \frac{\hat{\mathbf{u}}^T\hat{\mathbf{e}}}{\hat{\mathbf{u}}^T\overline{\mathbf{x}}} + (\kappa-\tau).$
\end{proof}

A result not much different from \eqref{marca1} but more restrictive appear in \cite{CSZ2010}.

\bigskip

Considering a graph $G$, a \textit{star set} for an eigenvalue $\mu \in \sigma(G)$, with multiplicity $m_G(\mu)$, is a vertex subset $X$
such that $|X|=m_G(\mu)$ and the graph $G-X$ does not have $\mu$ as an eigenvalue. The vertex complement subset $V(G) \setminus X$ is called a
co-star set and the graph $G-X$ is called a \textit{star complement} for $\mu$. The main properties of star sets and star complements
appear in \cite{CRS10}.

\begin{theorem}\cite{ACS2013}
Let $G$ be a graph and $X \subset V(G)$  a star (or co-star) set for the eigenvalue $\mu \in \sigma (G)$. If $X$ or $V(G) \setminus X$
is $(\kappa,\tau)$-regular in $G$, then $\mu$ is non-main if and only if $\mu = \kappa-\tau$.
\end{theorem}

\medskip\textbf{Acknowledgments}
The author was partially supported by the Portuguese Foundation for Science and Technology (\textquotedblleft FCT-Funda\c{c}\~{a}o
para a Ci\^{e}ncia e a Tecnologia\textquotedblright), through CIDMA - Center for Research and Development in Mathematics and
Applications, within project UID/MAT/04106/2013.

\end{document}